\documentclass[11 pt]{amsart}
\usepackage{amsmath}
\usepackage{amsfonts}
\usepackage{enumerate}
\usepackage{verbatim}
\usepackage[pdftex]{graphicx}
\DeclareGraphicsExtensions{.png,}
\newtheorem{lemma}{Lemma}
\newtheorem{theorem}{Theorem}

\newtheorem{corollary}{Corollary}

\newtheorem{defn}{Definition}
\numberwithin{equation}{section}
\numberwithin{defn}{section}
\numberwithin{prop}{section}
\numberwithin{lemma}{section}

\newcommand{\R}{\mathbb{R}}
\newcommand{\Prob}{\mathbb{P}}

\newcommand{\N}{\mathbb{N}}
\newcommand{\Q}{\mathbb{Q}}

\newcommand{\p}{\mathbf{p}}

\newcommand{\W}{\mathcal{W}}
\newcommand{\B}{\mathcal{B}}

\newcommand{\var}{\text{var}}

\newcommand{\D}{\mathcal{D}}
\newcommand{\M}{\mathcal{M}}

\renewcommand{\dim}{\mbox{dim}_{\mathcal{H}}}

\title[Multifractal Analysis on Lalley-Gatzouras Repellers]{Multifractal Analysis for Birkhoff Averages on Lalley-Gatzouras Repellers}
\author{Henry WJ Reeve}
\address{Henry WJ Reeve\\Department of Mathematics\\ The University of Bristol\\
University Walk\\Clifton\\ Bristol\\BS8 1TW\\UK}
\email{henrywjreeve@googlemail.com}
\begin{document}

\thanks{The author would like to thank his supervisor Thomas Jordan for all his help in preparing this paper. This project was started in the Instytut Matematyczny PAN and I would like to thank the Institute, especially Micha\l{} Rams and Feliks Przytycki, for their kind hospitality. I would also like to thank the Engineering and Physical Sciences Research Council and the Conformal Structures and Dynamics network for their financial support.
 }

\begin{abstract}
We consider the multifractal analysis for Birkhoff averages of continuous potentials on a class of non-conformal repellers corresponding to the self-affine limit sets studied by Lalley and Gatzouras. A conditional variational principle is given for the Hausdorff dimension of the set of points for which the Birkhoff averages converge to a given value. This extends a result of Barral and Mensi to certain non-conformal maps with a measure dependent Lyapunov exponent.
\end{abstract}
\maketitle
\section{Introduction and statement of results}
In this paper we consider the multifractal analysis of Birkhoff averages. Let $\Lambda$ be a repeller for a planar map $T: \R^2 \rightarrow \R^2$. Given a continuous potential $\varphi: \Lambda \rightarrow \R$ and $\alpha \in \R$ we are interested in the set of those points in the repeller for which the Birkhoff average converges to $\alpha$
\begin{equation}
\Lambda^{\varphi}_{\alpha}:= \left\lbrace x \in \Lambda : \lim_{n \rightarrow \infty} \frac{1}{n}\sum_{k=0}^{n-1} \varphi(T^k(x))=\alpha\right\rbrace.
\end{equation}
In particular we would like to understand how the Hausdorff dimension $\dim$ of $\Lambda^{\varphi}_{\alpha}$ varies as a function of $\alpha$,
\begin{equation} \label{Birkhoff Spectrum}
\alpha \mapsto \dim \Lambda^{\varphi}_{\alpha}.
\end{equation}
When $T$ is conformal and hyperbolic the function $\alpha \mapsto \dim \Lambda^{\varphi}_{\alpha}$ is well understood (see Pesin and Weiss \cite{Pesin Weiss Birkhoff}, Fan Feng and Wu \cite{Fan Feng Wu}, Barriera and Saussol \cite{Barriera Saussol} and Olsen \cite{Olsen Multifractal 1} for increasingly general results). However, in the non-conformal setting much less is known. Jordan and Simon \cite{Jordan Simon} gave a variational formula for $\dim \Lambda^{\varphi}_{\alpha}$ for typical members of families of piecewise diagonal maps. Barral and Mensi \cite{Barral Mensi} and Barral and Feng \cite{Barral Feng} give a precise formula for $\dim \Lambda^{\varphi}_{\alpha}$ in the setting of Bedford \cite{Bedford} and McMullen \cite{McMullen}.

\pagebreak
We shall prove a conditional variational principle for $\dim \Lambda^{\varphi}_{\alpha}$
for a more general class of piecewise affine maps $T:\R^2 \rightarrow \R^2$ with repellers $\Lambda$ corresponding to the self-affine limit sets studied by Lalley and Gatzouras in \cite{Gatzouras Lalley}.

\begin{figure}
\includegraphics[width=115mm]{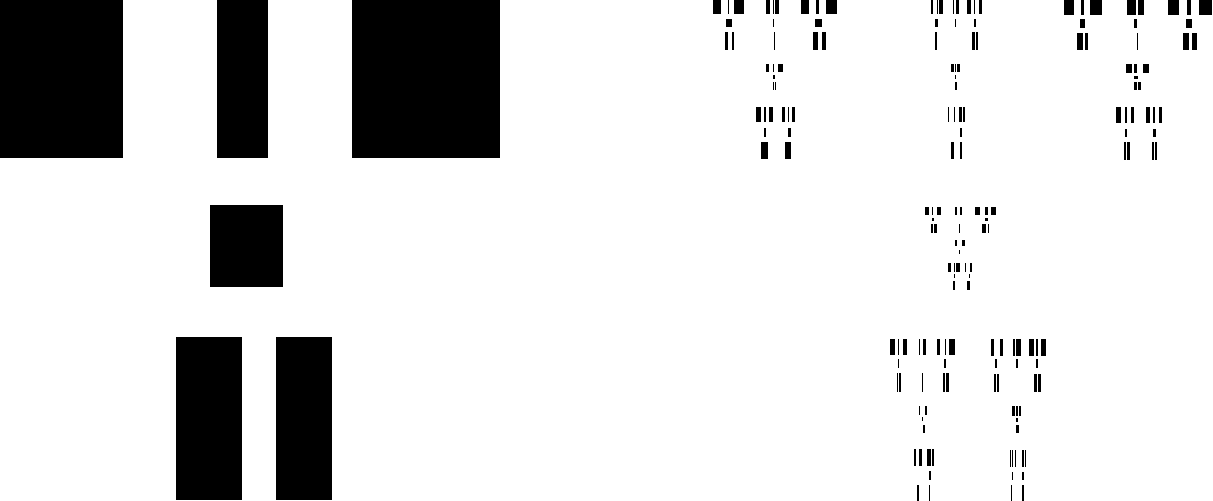}
\caption{A representation of a Lalley-Gatzouras system (left) and the corresponding limit set (right).}
\end{figure}

\begin{defn}[Lalley-Gatzouras Systems] \label{GL IFS def}Suppose we have some index set $\mathcal{D} = \left\lbrace  (i,j) : 1 \leq i \leq p \text{ and } 1 \leq j \leq m_i \right\rbrace$ and for each $(i,j)\in \D$ we define an affine contraction of the form
\begin{equation} S_{ij} (x)= \left(  \begin{array}{ll}
a_{ij} & 0 \\
0 & b_j             \end{array}\right) x + \left(
\begin{array}{l}
c_{ij} \\
d_i             \end{array}\right) \text{ for } x \in [0,1]^2
\end{equation}
where $a_{ij}$, $b_{i}$, $c_{ij}$, $d_{i}$ are fixed members of $[0,1]$, with $b_i$ and $d_i$ depending only on
$i$. Suppose that for each $(i,j)\in \D$ we have $0<a_{ij}\leq b_i<1$. We stipulate that
 $0\leq d_1 \leq d_2 \leq \cdots \leq d_p <1$ with $d_{i+1}-d_i
\geq b_i$ and $b_p+d_p \leq 1$ and for each $i$, $0\leq c_{i1} \leq c_{i2} \leq \cdots \leq c_{im_i} <1$ with
$c_{i (j+1)} -c_{ij}\geq a_{ij}$ and $a_{im_i}+c_{im_i} \leq 1$. We shall refer to a family of affine maps
$\left(S_{ij} \right)_{(i,j)\in \mathcal{D}}$, formed in this way, as a Lalley-Gatzouras system.
\end{defn}

Let $\Sigma:= \D^{\N}$ and $\Sigma_v:=\{1,\cdots,p\}^\N$ be full shift spaces with corresponding left shift
operators denoted by $\sigma: \Sigma \rightarrow \Sigma$ and $\sigma_v: \Sigma_v \rightarrow \Sigma_v$,
respectively. Given $\omega\in \Sigma$ and $n\in\N$ we let $\omega|n\in\D$ denote the finite string consisting of
the first $n$ terms of $\omega$. We define $\pi:\Sigma \rightarrow \Sigma_v$ by $ \pi: ((i_n,j_n))_{n\in\N}
\mapsto (i_n)_{n\in\N} $. Given $n\in \N$ we let $\M_{\sigma^n}(\Sigma)$ denote the set of Borel probability
measures supported on $\Sigma$ which are invariant under $\sigma^n$ and $\B_{\sigma^n}(\Sigma)$ the set of
Bernoulli measures with respect to $\sigma^n$. Similarly, we let $\M_{\sigma_v^n}(\Sigma_v)$ denote the set of
$\sigma_v^n$-invariant measures and  $\B_{\sigma_v^n}(\Sigma_v)$ the set of $\sigma_v^n$-Bernoulli measures. Note
that if $\mu \in \M_{\sigma^n}(\Sigma)$ then $\pi(\mu):=\mu \circ \pi^{-1} \in \M_{\sigma_v^n}(\Sigma_v)$ and if
$\mu \in \B_{\sigma^n}(\Sigma)$ then $\pi(\mu)\in \B_{\sigma_v^n}(\Sigma_v)$. Given $n \in \N$ and $\mu \in
\M_{\sigma^n}(\Sigma)$ we define corresponding Lyapunov exponents by
\begin{equation}
\lambda(\mu, \sigma^n):=-\int \log a_{\omega_1} \cdots a_{\omega_n} d\mu(\omega)\end{equation} and
\begin{equation}
\lambda^v(\mu, \sigma^n):=-\int \log b_{i_1} \cdots b_{i_n} d\mu(\omega).
\end{equation}

We also let $h(\mu,\sigma^n)$ denote the Kolmogorov-Sinai entropy of $\mu$ with respect to $\sigma^n$ and $h^v(\mu,\sigma^n)$ the Kolmogorov-Sinai entropy of $\pi(\mu)$ with respect to $\sigma_v^n$.
\newline
Given $n\in \N$ and $\mu \in \M_{\sigma^n}(\Sigma)$ we define
\begin{equation}
 D^n_{LY}(\mu) =\frac{h(\mu,\sigma^n)}{\lambda(\mu,\sigma^n)} +\bigg(\frac{1}{\lambda^v(\mu,\sigma^n)}-\frac{1}{\lambda(\mu,\sigma^n)}\bigg)h^v(\mu,\sigma^n);
\end{equation}
cf. Ledrappier and Young \cite{LYMED2} Corollary D. We write $D^1_{LY}(\mu)$ as $D_{LY}(\mu)$.
\newline
Let $C(\Sigma)$ denote the set of continuous potentials $\varphi: \Sigma \rightarrow \R$. Given $\varphi \in C(\Sigma)$ and $\alpha \in \R$ we let
\begin{equation}
\Sigma^{\varphi}_{\alpha}:= \left\lbrace \omega \in \Sigma : \lim_{n \rightarrow \infty}
\frac{1}{n}\sum_{l=0}^{n-1} \varphi(\sigma^l(\omega))=\alpha\right\rbrace.
\end{equation}
Let $\alpha_{\min}(\varphi):= \inf\left\lbrace \int \varphi d \mu: \mu \in \M_{\sigma}(\Sigma)\right\rbrace$ and
$\alpha_{\max}(\varphi):= \sup\left\lbrace \int \varphi d \mu: \mu \in \M_{\sigma}(\Sigma)\right\rbrace$. It is
easy to check that $[\alpha_{\min}(\varphi), \alpha_{\max}(\varphi)]=\left\lbrace \alpha \in \R:
\Sigma_{\alpha}^{\varphi}\neq \emptyset \right\rbrace$ (apply \cite{Walters} Theorems  1.14 and 6.9). Given a
potential $\varphi:\Sigma\rightarrow \R$, we define for each $k\in \N$ the $k$th average potential $A_k
(\varphi):\Sigma \rightarrow \R$ by $A_k (\varphi):=\frac{1}{k}\sum_{l=0}^{k-1} \varphi\circ \sigma^l$ and the $k$
variance $\var_k(\varphi):=\sup\left\lbrace |\varphi(\omega)-\varphi(\tau)|: \omega_l=\tau_l\text{ for
}l=1,\cdots,k\right\rbrace$.
\newline
Let $\chi^h:[0,1]^2\rightarrow [0,1]$ denote the horizontal projection given by $(x_1,x_2)\mapsto x_1$ and
$\chi^v:[0,1]^2\rightarrow [0,1]$ the vertical projection $(x_1,x_2)\mapsto x_2$. For each $(i,j) \in \D$ we let
$f_{ij}$ denote the affine map $x\mapsto a_{ij}x+c_{ij}$ and $g_{i}$ denote the affine map $x\mapsto
b_{i}x+d_{i}$. It follows that $f_{ij} \circ \chi^h = \chi^h \circ S_{ij}$ and $g_{i} \circ \chi^v = \chi^v \circ
S_{ij}$.
\newline
Given a finite string $\eta=\eta_1\cdots\eta_n\in \D^n$ we let $S_{\eta}:=S_{\eta_1}\circ\cdots\circ S_{\eta_n}$
and $f_{\eta}:=f_{\eta_1}\circ\cdots\circ f_{\eta_n}$. Similarly given $\zeta=\zeta_1\cdots\zeta_n\in
\{1,\cdots,p\}^n$ we let $g_{\zeta}:=g_{\zeta_1}\circ\cdots\circ g_{\zeta_n}$. There is a natural projection $\Pi:
\Sigma \rightarrow \R^2$ given by
\begin{equation}
\Pi(\omega)= \lim_{n\rightarrow \infty} S_{\omega|n}([0,1]^2).
\end{equation}
Define $\Lambda:= \Pi(\Sigma)$ and for each $\varphi \in C(\Sigma)$ and $\alpha \in \R$ we let  $\Lambda^{\varphi}_{\alpha}:=\Pi(\Sigma^{\varphi}_{\alpha})$. Note that $\Lambda$ is the unique non-empty compact set satisfying $\Lambda= \bigcup_{(i,j)\in \D} S_{ij}(\Lambda)$.
It was shown by Lalley and Gatzouras in \cite{Gatzouras Lalley} that, \vspace{2mm}

\begin{flushleft}
\textbf{Theorem 1.1.} (Lalley and Gatzouras, 1992)\end{flushleft}
\begin{equation*}
 \dim{\Lambda}= \sup \left\lbrace D_{LY}(\mu) : \mu \in \mathcal{B}_{\sigma}(\Sigma) \right\rbrace.
 \end{equation*}

\vspace{2mm}
The central purpose of this paper is to prove Theorem \ref{main}.
\begin{theorem}\label{main} Suppose $\varphi \in C(\Sigma)$. Then for all $\alpha \in [\alpha_{\min}(\varphi), \alpha_{\max}(\varphi)]$ we have
\[\dim{\Lambda^{\varphi}_{\alpha}}= \sup \left\lbrace D_{LY}(\mu): \mu \in \M_{\sigma}(\Sigma), \int \varphi d\mu =\alpha \right\rbrace.  \]
In particular $\alpha \mapsto\dim{\Lambda^{\varphi}_{\alpha}}$ is continuous on $[\alpha_{\min}(\varphi), \alpha_{\max}(\varphi)]$.
\end{theorem}
Corresponding to each Lalley-Gatzouras IFS satisfying $S_{i_1j_1}([0,1]^2)\cap S_{i_2j_2}([0,1]^2)=\emptyset$ for
$(i_1,j_1)\neq(i_2,j_2)\in \D$ there is an associated piecewise affine planar map $T:\R^2 \rightarrow \R^2$. $T$
is the unique orientation preserving piecewise affine map which sends each rectangle $R_{ij}:=[c_{ij},
c_{ij}+a_{ij}]\times  [d_{i}, d_{i}+b_{i}]$ to $[0,1]^2$ and leaves the rest of the plane fixed. The set
\begin{equation}
\Lambda=\left\lbrace x \in \R^2: T^n(x)\in \bigcup_{(i,j) \in \D} R_{ij} \text{ for all }n \geq 0\right\rbrace
\end{equation}
is a repeller for $T$, known as a Lalley-Gatzouras repeller. The dynamical interest in Theorem \ref{main} is that
it allows us to give the multifractal analysis for Birkhoff averages (see (\ref{Birkhoff Spectrum}) above) for
maps $T: \R^2 \rightarrow \R^2$ of this form.\newline Note that the special case of Theorem \ref{main} in which
each of the maps $S_{ij}$ is a similarity may be deduced from Olsen \cite{Olsen Multifractal 1} Theorem 1.
Moreover, the special case in which there exists constants $a,b$ for which $a_{ij}=a$ and $b_i=b$ for all
$(i,j)\in \D$ was solved by Barral and Mensi in \cite{Barral Mensi} using a weighted version of the thermodynamic
formalism. However, when we are in the non-conformal setting with measure dependent Lyapunov exponents the
thermodynamic formalism does not apply and a different approach is required. For the lower bound, we combine ideas
from Lalley and Gatzouras \cite{Gatzouras Lalley} and Gelfert and Rams \cite{Gelfert Rams}. For the upper bound,
we begin by adapting a technique from Bara\'{n}ski \cite{Baranski} to prove the result for locally constant
potentials before applying an approximation argument to obtain the result in full generality.\newline The paper is
structured as follows. In section \ref{Dimension Lemmas Sec} we recall the notion of an approximate square,
demonstrating how they may be used to give dimension estimates for projections of subsets of the symbolic space.
In section \ref{Proof of the lower bound} we prove the lower bound and in section \ref{Proof of the upper bound}
we prove the upper bound. We conclude with some remarks and an open question.

\section{Dimension Lemmas}\label{Dimension Lemmas Sec}
An estimate for Hausdorff dimension is obtained by finding optimal coverings. In the conformal setting it ordinarily suffices to consider families of projections of cylinder sets. However, in the non-conformal setting the geometric distortion resulting from a difference in expansion between the strong and the weak unstable foliation means that coverings of this form will be highly non-optimal. Instead we follow McMullen \cite{McMullen} and Lalley and Gatzouras \cite{Gatzouras Lalley} in using approximate squares for this purpose.\newline
We define for each $\omega \in \Sigma$ and $n \in \N$
\begin{equation}
L_n(\omega):=\min \left\lbrace l \geq 1 : \prod^l_{\nu=1} a_{i_{\nu} j_{\nu}} \leq \prod^n_{\nu=1} b_{i_{\nu}}\right\rbrace .
\end{equation}
Note that this implies
\begin{equation} \label{Lyapunov Ratios}
a_{\min} < \frac{\prod_{\nu=1}^{L_n(\omega)} a_{i_{\nu} j_{\nu}}}{\prod_{\nu=1}^n b_{i_{\nu}}} \leq 1 .
\end{equation}
Given $(\omega_{\nu})_{\nu=1}^n=((i_{\nu},j_{\nu}))_{\nu =1}^n\in \D^n$ we let
\begin{equation}
[\omega_{1} \cdots \omega_n]:=\{ \omega' \in \Sigma: \omega'_{\nu}= \omega_{\nu} \text{ for } \nu=1, \cdots, n\}
\end{equation}
and
\begin{equation}
[i_{1} \cdots i_n]:=\{ \omega' \in \Sigma: i'_{\nu}= i_{\nu} \text{ for } \nu=1, \cdots, n\}.
\end{equation}
Given $\omega= ((i_{\nu}, j_{\nu}))_{\nu =1}^{\infty}\in \Sigma$ we let $B_n(\omega)$ denote the $n$th approximate symbolic square, \begin{equation}
B_n(\omega):=[\omega_1 \cdots \omega_{L_n(\omega)} ]\cap \sigma^{-L_n(\omega)}[i_{L_n(\omega)+1} \cdots i_n].
\end{equation}
We let $\Delta_n(\omega)$ denote the approximate square corresponding to $B_n(\omega)$, defined by
\begin{equation}
\Delta_n(\omega):=f_{\omega|L_n(\omega)}([0,1])\times g_{\mathbf{i}|n}([0,1]).
\end{equation}
Note that for each $\omega \in \Sigma$ and $n\in \N$, $\Pi(B_n(\omega))\subseteq \Delta_n(\omega)$ and for all $\omega' \notin B_n(\omega)$, $\text{int}(\Delta_n(\omega))\cap \text{int}(\Delta_n(\omega'))=\emptyset$.

We say that the digit set $\D$ is two-dimensional if there exists $(i_1,j_1),(i_2,j_2)\in \D$ with $i_1=i_2$ and $j_1\neq j_2$ and there exists $(i_3,j_3),(i_4,j_4)\in \D$ with $i_3\neq i_4$.
Define for each $d\in \D$
\begin{equation}
R^d_n(\omega):= \min\left\lbrace l>n: \omega_l=d\right\rbrace-n
\end{equation}
and
\begin{equation}
R_n(\omega):= \max\left\lbrace R^d_n(\omega): d\in \D \right\rbrace.
\end{equation}

\begin{lemma} \label{SymbDimPDim} Let $\mu$ be a finite Borel measure on $\Sigma$ and $\nu:= \mu \circ \Pi^{-1}$ the corresponding projection on $\Lambda$.
\begin{enumerate}
\item [(i)] Suppose $\D$ is two-dimensional. Then for all $x=\Pi(\omega) \in \Lambda$ with $\lim_{n\rightarrow \infty} \frac{R_n(\omega)}{n}=0$, \[\liminf_{r \rightarrow 0}\frac{\log \nu( B(x,r))}{\log r}\geq \liminf_{n \rightarrow \infty}\frac{\log \mu(B_n(\omega))}{\log \prod_{\nu=1}^n b_{\nu}}.\]
\item [(ii)] For all $x=\Pi(\omega) \in \Lambda$, \[\liminf_{r \rightarrow 0}\frac{\log \nu(B(x,r))}{\log r}\leq \liminf_{n \rightarrow \infty}\frac{\log \mu(B_n(\omega))}{\log \prod_{\nu=1}^n b_{\nu}}.\]
\end{enumerate}
\end{lemma}
\begin{proof} By Lipchitz equivalence it suffices to prove the lemma with respect to the maximum norm on $\R^2$.\newline
To prove \textbf{(i)} we first suppose that $\D$ is two-dimensional and fix  $x=\Pi(\omega) \in \Lambda$ with
$\lim_{n\rightarrow \infty} \frac{R_n(\omega)}{n}=0$. Now the horizontal projection $\chi^h(x)$ is contained
within $f_{\omega|L_n(\omega)+R_{L_n(\omega)}(\omega)}([0,1])$. Since $\D$ is two dimensional, there exists
$d_1=(i_1,j_1)\in\D$ and $d_2=(i_2,j_2)\in \D$ such that $i_1=i_1$ and $j_1\neq j_2$ and without loss of
generality we may suppose that $f_{d_1}(1)\leq f_{d_2}(0)$. By the definition of $R_n(\omega)$ both $d_1$ and
$d_2$ occur within the finite string $\eta_0:=\omega_{L_n(\omega)+1}\cdots\omega_{L_n(\omega)+R_{L_n(\omega)}}$.
Now let $\eta_1$ be the string $\eta_0$ but with an extra occurrence of $d_1$ in place of the first occurrence of
$d_2$ and similarly let $\eta_2$ be $\eta_0$ but with an extra occurrence of $d_2$ in place of the first
occurrence of $d_1$. Now consider the three intervals \begin{eqnarray*} f_{\omega|L_n(\omega)}\circ
f_{\eta_1}([0,1])\hspace{1cm}f_{\omega|L_n(\omega)}\circ f_{\eta_0}([0,1])\hspace{1cm} f_{\omega|L_n(\omega)}\circ
f_{\eta_2}([0,1]).
\end{eqnarray*}
Each interval is of width at least $\prod_{\nu=1}^{L_n(\omega)} a_{i_{\nu} j_{\nu}}\times a_{\min}^{R_{L_n(\omega)}(\omega)}\geq \prod_{\nu=1}^{n} b_{i_{\nu}}\times a_{\min}^{R_{L_n(\omega)}(\omega)+1}$ and is contained within the interval $f_{\omega|L_n(\omega)}([0,1])$. Since the three intervals have disjoint interior and $\chi^h(x)$ is contained within the middle one, it follows that $\chi^h(x)$ is at least $\prod_{\nu=1}^{n} b_{i_{\nu}}\times a_{\min}^{R_{L_n(\omega)}(\omega)+1}$ away from both the left and the right end points of $f_{\omega|L_n(\omega)}([0,1])$. Similarly using the existence of $(i_3,j_3),(i_4,j_4)\in \D$ with $i_3\neq i_4$ we may show that $\chi^v(x)$ is at least $\prod_{\nu=1}^{n} b_{i_{\nu}}\times a_{\min}^{R_{n}(\omega)}$ away from both the left and the right end points of $g_{\mathbf{i}|n}([0,1])$. Thus, we have
\begin{equation}
\Pi\left(B\left(x,\prod_{\nu=1}^{n} b_{i_{\nu}}\times a_{\min}^{\max\{R_{L_n(\omega)}(\omega)+1,R_n(\omega)\}}\right)\right)\subseteq \Delta_n(\omega).
\end{equation}
If we let $n_r:=\max\left\lbrace n\in \N:\prod_{\nu=1}^{n} b_{i_{\nu}}\times
a_{\min}^{\max\{R_{L_n(\omega)}(\omega)+1,R_n(\omega)\}}>r\right\rbrace$ then
$\Pi^{-1}\left(B(x,r)\right)\subseteq B_{n_r}(\omega)$ whilst $\prod_{\nu=1}^{n_r} b_{i_{\nu}}\times
a_{\min}^{\max\{R_{L_{n_r+1}(\omega)}(\omega),R_{n_r+1}(\omega)\}+2}\leq r$.\newline Hence,
\begin{eqnarray*}
\frac{\log \nu(B(x,r))}{\log r} \geq \frac{\log \mu(B_{n_r}(\omega))}{\log \prod_{\nu=1}^{n_r}
b_{i_{\nu}}+{\max\{R_{L_{n_r+1}(\omega)}(\omega),R_{n_r+1}(\omega)\}+2}\log a_{\min}}.
\end{eqnarray*}
Since $\lim_{n\rightarrow \infty}\frac{R_n(\omega)}{n}=0$ and $\liminf_{n\rightarrow \infty}\frac{L_n(\omega)}{n}>0$, \textbf{(i)} follows.

\vspace{.1cm} For \textbf{(ii)} we begin by fixing $x=\Pi(\omega)\in \Lambda$. For each $n\in \N$ the image
$\Pi(B_n(\omega))$ contains $x$ and has diameter not exceeding $\prod_{\nu=1}^n b_{i_{\nu}}$. Thus
$\Pi(B_n(\omega)) \subseteq B(x; \prod_{\nu=1}^n b_{i_{\nu}})$ and hence
\begin{eqnarray} \frac{ \log \nu(B(x; \prod_{\nu=1}^n b_{i_{\nu}}))}{\log \prod_{\nu=1}^n b_{i_{\nu}}}
 \leq \frac{ \log \mu(B_n(\omega))}{\log \prod_{\nu=1}^n  b_{i_{\nu}}}.
\end{eqnarray}
Letting $n\rightarrow \infty$ proves the lemma.

\end{proof}

Recall the following results from geometric measure theory.
\begin{lemma}\label{Dimension Lemmas} Let $\nu$ be a finite Borel measure on some metric space $X$.
\begin{enumerate}
 \item Suppose we have $J\subseteq X$ with $\nu(J)>0$ such that for all $x \in J$
 \[\liminf_{r \rightarrow 0}\frac{\log \nu( B(x,r))}{\log r} \geq d.\]Then $\dim J \geq d$.
 \item Suppose we have $J\subseteq X$ such that for all $x \in J$
 \[\liminf_{r \rightarrow 0}\frac{\log \nu(B(x,r))}{\log r} \leq d.\]Then $\dim J \leq d$.
\end{enumerate}
\end{lemma}
\begin{proof}
See \cite{Falconer Techniques} Proposition 2.2.
\end{proof}
\begin{lemma}\label{SymbolicDimensionLemmas} Let $\mu$ be a finite Borel measure on $\Sigma$.
\begin{enumerate}
 \item Suppose $\D$ is two dimensional and we have $S\subseteq \Sigma$ with $\mu(S)>0$ such that for all $\omega \in S$ \[\hspace{.75cm}\lim_{n\rightarrow \infty}\frac{R_n(\omega)}{n}=0
 \hspace{.75cm}\text{ and }\hspace{.75cm}
\liminf_{n \rightarrow \infty}\frac{\log\mu( B_n(\omega))}{\log \prod_{\nu=1}^n b_{\nu}} \geq d.\]Then $\dim \Pi(S) \geq d$.
 \item Suppose we have $S\subseteq \Sigma$ such that for all $\omega \in S$ \[\liminf_{n \rightarrow \infty}\frac{\log \mu( B_n(\omega))}{\log \prod_{\nu=1}^n b_{\nu}} \leq d.\]Then $\dim \Pi(S) \leq d$.
\end{enumerate}
\end{lemma}
\begin{proof}
Combine Lemma \ref{SymbDimPDim} with Lemma \ref{Dimension Lemmas}.
\end{proof}
Lemma \ref{SymbolicDimensionLemmas} (i) will be used for the lower bound and Lemma \ref{SymbolicDimensionLemmas} (ii) for the upper bound.

\section{Proof of the lower bound}\label{Proof of the lower bound}
The desired lower bound is a supremum of $D_{LY}(\mu)$ over certain invariant measures. In order to obtain a
dimension estimate we need to apply Birkhoff's ergodic theorem, so we must approximate invariant measures by
ergodic ones. However, these approximations have an error term which may cause them to be supported by the wrong
level set. In order to obtain the correct lower bound we follow the approach of Gelfert and Rams in \cite{Gelfert
Rams} and construct a measure which behaves asymptotically like increasingly accurate ergodic approximations to a
given invariant measure. \newline Throughout the proof of the lower bound we fix some
$\alpha\in[\alpha_{\min}(\varphi), \alpha_{\max}(\varphi)]$ and some $\mu \in \M_{\sigma}(\Sigma)$ satisfying
$\int \varphi d\mu=\alpha$. We shall show that $\dim \Lambda_{\alpha}^{\varphi} \geq D_{LY}(\mu)$.  \newline For
each $k\in\N$ we let $\tilde{\mu}_k \in B_{\sigma^k}(\Sigma)$ denote the $k$-th level approximation of $\mu$. That
is, given a cylinder $[\omega_1\cdots \omega_{nk}]$ of length $nk$ we let
\begin{equation}
\tilde{\mu}_k([\omega_1\cdots\omega_{nk}]):=\prod_{l=0}^{n-1}\mu([\omega_{lk+1}\cdots\omega_{lk+k}]).
\end{equation}

\begin{lemma} \label{k Bernoulli convergence} There exists a sequence $\{\mu_k\}$ of  measures $\mu_k \in B_{\sigma^k}(\Sigma)$ satisfying,
\begin{enumerate}
\vspace{4mm}
\item [(i)] $\lim_{k\rightarrow\infty} \frac{1}{k} h(\mu_k,\sigma^k)=h(\mu,\sigma)$
\vspace{4mm}
\item [(ii)] $\lim_{k\rightarrow\infty} \frac{1}{k} h^v(\mu_k,\sigma^k)=h^v(\mu,\sigma)$
\vspace{4mm}
\item [(iii)] $\lim_{k\rightarrow\infty}\frac{1}{k}\lambda(\mu_k,\sigma^k)=\lambda(\mu,\sigma)$
\vspace{4mm}
\item[(iv)] $\lim_{k\rightarrow\infty}\frac{1}{k}\lambda^v(\mu_k,\sigma^k)=\lambda^v(\mu,\sigma)$
\vspace{4mm}
\item [(v)] $\lim_{k\rightarrow\infty} \int A_k \varphi d\mu_k=\alpha.$
\vspace{4mm}
\item [(vi)] For each $k\in \N$ and $(\omega_1,\cdots,\omega_k)\in \D^k$ we have $\mu_k([\omega_1\cdots\omega_k])>0$.
\vspace{4mm}
\end{enumerate}
\end{lemma}
\begin{proof} We begin by observing that parts $(i)-(v)$ are satisfied by $\{\tilde{\mu}_k\}$, the sequence of $k$th level approximations to $\mu$. Indeed parts $(i)$ and $(ii)$ follow from the Kolmogorov-Sinai theorem (see \cite{Walters} Theorem 4.18). Since $\mu$ is $\sigma$ invariant with $\mu$ and $\tilde{\mu_k}$ agreeing on cylinders of length $k$ we have $\lambda(\mu,\sigma)=1/k\lambda(\mu,\sigma^k)=1/k\lambda(\tilde{\mu}_k,\sigma^k)$. $(iv)$ may be proved similarly. To see $(v)$ we note that by $\sigma$ invariance of $\mu$, $\int A_k (\varphi) d\mu=\alpha$ and since $\mu$ and $\tilde{\mu}_k$ agree on cylinders of length $k$ we have $\big|\int A_k (\varphi) d\tilde{\mu}_k-\int A_k (\varphi) d\mu\big|\leq \var_k A_k(\varphi)$. Moreover, by the continuity of $\varphi$ $\var_k A_k(\varphi)\rightarrow 0$ as $k\rightarrow \infty$.\newline
To obtain $\{\mu_k\}$ with $\mu_k\in B_{\sigma^k}(\Sigma)$ satisfying $(vi)$ in addition to conditions $(i)-(v)$
we perturb each $\tilde{\mu}_k$ by a small amount to obtain $\mu_k$ with $\mu_k([\omega_1\cdots\omega_k])>0$ for
each $(\omega_1,\cdots,\omega_k)\in \D^k$ whilst using continuity to insure that
\begin{eqnarray}
 \big|h(\mu_k,\sigma^k)-h(\tilde{\mu}_k,\sigma^k)\big|&<&\frac{1}{k}\\
 \big|h^v(\mu_k,\sigma^k)-h^v(\tilde{\mu}_k,\sigma^k)\big|&<&\frac{1}{k}\\
 \big|\lambda(\mu_k,\sigma^k)-\lambda(\tilde{\mu}_k,\sigma^k)\big|&<&\frac{1}{k}\\
 \big|\lambda^v(\mu_k,\sigma^k)-\lambda^v(\tilde{\mu}_k,\sigma^k)\big|&<&\frac{1}{k}\\
 \big|\int A_k \varphi d\mu_k-\int A_k \varphi d\tilde{\mu}_k\big|&<&\frac{1}{k}.
\end{eqnarray}
\end{proof}
Now choose $\delta_q>0$ for each $q \in \N$ in such a way that $\prod_{q=1}^{\infty}(1-\delta_q)>0$.

\begin{lemma} \label{ q lemma}
For each $q \in \N$ we may choose $k(q),  B(q), N(q) \in \N$ and $S_q \subseteq \Sigma$ with $\mu_{k(q)}(S_q)>1-\delta_q$ such that for all $\omega =(i_{\nu},j_{\nu})_{\nu \in \N} \in S_q$ and $n \in \N$ we have
\begin{enumerate}
\vspace{4mm}
\item [(i)] $\frac{1}{nk(q)}\log \mu_{k(q)}([\omega_1\cdots \omega_{nk(q)}])>-B(q)$
\vspace{4mm}
\item [(ii)] $\frac{1}{nk(q)}\log \mu_{k(q)}([i_1\cdots i_{nk(q)}])>-B(q)$
\vspace{4mm}
\item [(iii)] $\frac{R_n(\omega)}{n}<B(q)$
\vspace{4mm}
\end{enumerate}
and for all $n \geq N(q)$ we have
\begin{enumerate}
\vspace{4mm}
\item [(iv)] $ \big| \frac{1}{n k(q)} \log \mu_{k(q)}([\omega_1 \cdots \omega_{n k(q)}])+h(\mu,\sigma)\big|< \frac{1}{q}$
\vspace{4mm}
\item [(v)] $ \big| \frac{1}{n k(q)} \log \mu_{k(q)}([i_1 \cdots i_{n k(q)}])+h^v(\mu,\sigma)\big|< \frac{1}{q}$
\vspace{4mm}
\item [(vi)] $\big| \frac{1}{n k(q)} \log \prod_{\nu=1}^{nk(q)} a_{i_{\nu}j_{\nu}} +\lambda(\mu,\sigma)\big|< \frac{1}{q}$
\vspace{4mm}
\item[(vii)] $\big| \frac{1}{n k(q)} \log \prod_{\nu=1}^{nk(q)} b_{i_{\nu}} +\lambda^v(\mu,\sigma)\big|< \frac{1}{q}$
\vspace{4mm}
\item [(viii)] $\big| \frac{1}{n k(q)} \sum_{l=0}^{nk(q)-1} \varphi (\sigma^l \omega)-\alpha\big|< \frac{1}{q}$
\vspace{4mm}
\item [(ix)] $\frac{R_{n}(\omega)}{n} < \frac{1}{q}.$
\end{enumerate}
\vspace{4mm}
\end{lemma}

\begin{proof}
Fix $q\in \N$. By Lemma \ref{k Bernoulli convergence} (i)-(v) we may choose $k(q) \in \N$ so that
\begin{eqnarray}
\label{k Bernoulli convergence1}\bigg|\frac{1}{k(q)}h(\mu_{k(q)},\sigma^{k(q)})-h(\mu,\sigma)\bigg|&<& \frac{1}{2q}\\
\label{k Bernoulli convergence2}\bigg|\frac{1}{k(q)}h^v(\mu_{k(q)},\sigma^{k(q)})-h^v(\mu,\sigma)\bigg|&<& \frac{1}{2q}\\
\label{k Bernoulli convergence3}\bigg| \int A_{k(q)} \varphi d\mu_{k(q)}-\alpha\bigg|&<& \frac{1}{2q}\\
\label{k Bernoulli convergence4}
\bigg|\frac{1}{k(q)}\lambda(\mu_{k(q)},\sigma^{k(q)})-\lambda(\mu,\sigma)\bigg|&<& \frac{1}{2q}\\
\label{k Bernoulli convergence5}
\bigg|\frac{1}{k(q)}\lambda^v(\mu_{k(q)},\sigma^{k(q)})-\lambda^v(\mu,\sigma)\bigg|&<& \frac{1}{2q}.
\end{eqnarray}

Noting that $\mu_{k(q)} \in \B_{\sigma^{k(q)}}(\Sigma)$ is ergodic with respect to $\sigma^{k(q)}$ we may apply
Birkhoff's ergodic theorem to obtain $\mu_{k(q)}$ almost everywhere convergences
\begin{eqnarray}
\label{entpcon} \lim_{n\rightarrow \infty}\frac{1}{n } \log \mu_{k(q)}([\omega_1 \cdots \omega_{n k(q)}])&=&-h(\mu_{k(q)},\sigma^{k(q)})\\
\label{entpconv} \lim_{n\rightarrow \infty}\frac{1}{n } \log \mu_{k(q)}([i_1 \cdots i_{n k(q)}])&=&-h^v(\mu_{k(q)},\sigma^{k(q)})\\
\label{lyapcon} \lim_{n\rightarrow \infty} \frac{1}{n} \log \prod_{\nu=1}^{nk(q)} a_{i_{\nu}j_{\nu}} &=&-\lambda(\mu_{k(q)},\sigma^{k(q)})\\
\label{lyapconv} \lim_{n\rightarrow \infty} \frac{1}{n} \log \prod_{\nu=1}^{nk(q)} b_{i_{\nu}} &=&-\lambda^v(\mu_{k(q)},\sigma^{k(q)})\\
\label{potentcon} \lim_{n\rightarrow \infty} \frac{1}{n} \sum_{l=0}^{nk(q)-1} \varphi (\sigma^l \omega)&=&\int
A_{k(q)} \varphi d\mu_{k(q)}.
\end{eqnarray}
and for each $(\tau_1,\cdots,\tau_{k(q)})\in \D^{k(q)}$ and $\mu_{k(q)}$ almost every $\omega\in \Sigma$
\begin{equation}\label{digitlistlimits}
\lim_{n\rightarrow\infty}\frac{\#\{l\in\{0,\cdots,n-1\}:\omega_{l{k(q)}+1}\cdots \omega_{l{k(q)}+{k(q)}}=\tau_1\cdots\tau_{k(q)}\}}{n}
\end{equation}
\begin{equation}
=\mu_{k(q)}([\tau_1\cdots\tau_{k(q)}])>0.
\end{equation}
For each of the limits (\ref{digitlistlimits}) to exist we must have
\begin{equation}
\lim_{n\rightarrow \infty} \frac{R^d_{nk(q)}(\omega)}{nk(q)}=0.
\end{equation}
Noting the definition of $R_{n}(\omega)$ along with the fact that $R_{nk(q)-l}(\omega)\leq R_{nk(q)}-l$ for $0\leq l\leq k(q)$ we have
\begin{equation} \label{R con}
\lim_{n\rightarrow \infty} \frac{R_{n}(\omega)}{n}=0.
\end{equation}
for $\mu_{k(q)}$ almost every $\omega \in \Sigma$.\newline By Egorov's theorem, we may take a set $S_q\subseteq
\Sigma$ with $\mu_{k(q)}(S_q)>1-\delta_q$ upon which each of the convergences (\ref{entpcon}), (\ref{entpconv}),
(\ref{lyapcon}), (\ref{lyapconv}), (\ref{potentcon}) and (\ref{R con}) is uniform. In particular, by taking $B(q)
\in \N$ sufficiently large we have
\begin{eqnarray}
 \frac{1}{nk(q)}\log \mu_{k(q)}([\omega_1\cdots \omega_{nk(q)}])&>&-B(q)\\
 \frac{1}{nk(q)}\log \mu_{k(q)}([i_1\cdots i_{nk(q)}])&>&-B(q)\\
\frac{R_{n}(\omega)}{n}&<&B(q)
\end{eqnarray}
for all $n \in \N$ and all $\omega \in S_q$. Moreover by taking $N(q) \in \N$ sufficiently large we have
\begin{eqnarray}
\bigg|\frac{1}{n k(q)} \log \mu_{k(q)}([\omega_1 \cdots \omega_{n k(q)}])+\frac{1}{k(q)}h(\mu_{k(q)},\sigma^{k(q)})\bigg|&<& \frac{1}{2q}\\
\bigg|\frac{1}{n k(q)} \log \mu_{k(q)}([i_1 \cdots i_{n k(q)}])+\frac{1}{k(q)}h^v(\mu_{k(q)},\sigma^{k(q)})\bigg|&<& \frac{1}{2q}\\
\bigg| \frac{1}{n k(q)} \log \prod_{\nu=1}^{nk(q)} a_{i_{\nu}j_{\nu}}+\frac{1}{k(q)}\lambda(\mu_{k(q)},\sigma^{k(q)})\bigg|&<& \frac{1}{2q}\\
\bigg| \frac{1}{n k(q)} \log \prod_{\nu=1}^{nk(q)} b_{i_{\nu}} +\frac{1}{k(q)}\lambda^v(\mu_{k(q)},\sigma^{k(q)})\bigg|&<& \frac{1}{2q}\\
\bigg| \frac{1}{n k(q)} \sum_{l=0}^{nk(q)-1} \varphi (\sigma^l \omega)-\int A_{k(q)} \varphi d\mu_{k(q)}\bigg|&<& \frac{1}{2q}\\
 \frac{R_{n}(\omega)}{n} &<& \frac{1}{q}
\end{eqnarray}
for all $n \geq N(q)$ and all $\omega \in S_q$. Combining these inequalities with the inequalities in (\ref{k Bernoulli convergence1}), (\ref{k Bernoulli convergence2}) and (\ref{k Bernoulli convergence3}) proves the lemma.
\end{proof}
We now construct our measure $\W$. First define a rapidly increasing sequence $(\gamma_q)_{q\in\N\cup\{0\}}$ of natural numbers by $\gamma_0=0$, $\gamma_1=1$ and for $q>1$ we let
\begin{equation}
\gamma_q:=q \gamma_{q-1}  \left(\prod_{l=1}^{q+1}N(l)\right)\left(\prod_{l=1}^{q+1}B(l)\right)\left(\prod_{l=1}^{q+1}k(l)\right)+\gamma_{q-1}.
\end{equation}
We now define a measure $\W$ on $\Sigma$ by first defining $\W$ on a semi-algebra of cylinders and then extending $\W$ to a Borel probability measure on $\Sigma$ via the Daniell-Kolmogorov consistency theorem (\cite{Walters} Theorem 0.5). Given a cylinder  $[\omega_1\cdots\omega_{\gamma_Q}]$ of length $\gamma_Q$ for some $Q \in \N$ we define
\begin{equation*}
\W([\omega_1\cdots\omega_{\gamma_Q}]):=\prod_{q=1}^Q\mu_{k(q)}([\omega_{\gamma_{q-1}+1}\cdots \omega_{\gamma_q}]).
\end{equation*}
Define $S\subseteq \Sigma$ by
\begin{equation}
S:=\bigcap_{q=1}^{\infty}\left\lbrace \omega\in \Sigma:[\omega_{\gamma_{q-1}+1}\cdots\omega_{\gamma_q}]\cap S_q\neq \emptyset\right\rbrace .
\end{equation}
\begin{lemma}\label{S>0}
$\W(S)>0$.
\end{lemma}
\begin{proof}$\W(S)\geq \prod_{q=1}^{\infty}\mu_q(S_q)>\prod_{q=1}^{\infty}(1-\delta_q)>0$.
\end{proof}

\begin{lemma}\label{Birkhoff Limits} For all $\omega \in S$ we have
\begin{enumerate}
\vspace{4mm}
\item [(i)] $\lim_{n\rightarrow \infty}\frac{1}{n}\log \prod_{\nu=1}^{n} a_{i_{\nu}j_{\nu}} =-\lambda(\mu,\sigma)$
\vspace{4mm}
\item[(ii)] $\lim_{n\rightarrow \infty}\frac{1}{n} \log \prod_{\nu=1}^{n} b_{i_{\nu}} =-\lambda^v(\mu,\sigma)$
\vspace{4mm}
\item [(iii)] $\lim_{n\rightarrow \infty}\frac{1}{n} \sum_{l=0}^{n-1} \varphi (\sigma^l \omega)=\alpha.$
\end{enumerate}

\end{lemma}

\begin{proof} We shall prove part (iii). The proofs for parts (i) and (ii) are similar.
Fix $\omega\in S$ and choose for each $q \in \N$ some $\tau^q \in \Sigma$ such that $\sigma^{\gamma_{q-1}} \tau_q
\in [\omega_{\gamma_{q-1}+1}\cdots\omega_{\gamma_q}]\cap S_q$. Given $n\in \N$ we choose $q_n$ so that
$\gamma_{q_n}\leq n$ is maximal. Since $\gamma_{q_n}-\gamma_{q_n-1}\geq N(q_n)$ and $\gamma_{q_n}-\gamma_{q_n-1}
\leq n$ we have
\begin{eqnarray} \label{B1'}
\bigg|\sum_{l=\gamma_{q_n-1}}^{\gamma_{q_n}-1} \varphi (\sigma^l \tau^{q_n})-(\gamma_{q_n}-\gamma_{q_n-1})\alpha\bigg|<\frac{n}{q_n}
\end{eqnarray}
by Lemma \ref{ q lemma} (viii). So by our choice of $\tau^q$ and $\gamma_{q_n}-\gamma_{q_n-1}\leq n$ we have
\begin{eqnarray} \label{B1}
\bigg|\sum_{l=\gamma_{q_n-1}}^{\gamma_{q_n}-1} \varphi (\sigma^l \omega)-(\gamma_{q_n}-\gamma_{q_n-1})\alpha\bigg|<\frac{n}{q_n}+\sum^{n-1}_{l=0}\var_{l}(\varphi).
\end{eqnarray}
By the construction of $(\gamma_q)_{q \in \N}$, $\gamma_{q_n-1}\leq \gamma_{q_n}/q_n \leq n/q_n$ and hence
\begin{eqnarray} \label{B2}
\bigg|\sum_{l=0}^{\gamma_{q_n-1}-1} \varphi (\sigma^l \omega)-\gamma_{q_n-1}\alpha\bigg|<\frac{n}{q_n}(||\varphi||_{\infty}+\alpha).
\end{eqnarray}
Now either $n-\gamma_{q_n}\geq N(q_n+1)$ or $n -\gamma_{q_n}\leq N(q_n+1)$. In the former case we reason as in (\ref{B1'}) and (\ref{B1}) to obtain
\begin{eqnarray} \label{B3}
\bigg|\sum_{l=\gamma_{q_n}}^{n-1} \varphi (\sigma^l \omega)-(n-\gamma_{q_n})\alpha\bigg|<\frac{n}{q_n}+\sum^{n-1}_{l=0}\var_{l}(\varphi).
\end{eqnarray}
In the latter case, by the construction of $(\gamma_q)_{q \in \N}$ we have $N(q_n+1) \leq \gamma_{q_n}/q_n \leq n/q_n$ and hence

\begin{eqnarray} \label{B4}
\bigg|\sum_{l=\gamma_{q_n}}^{n-1} \varphi (\sigma^l \omega)-(n-\gamma_{q_n})\alpha\bigg|<\frac{n}{q_n}(||\varphi||_{\infty}+\alpha).
\end{eqnarray}
Thus, by (\ref{B1}), (\ref{B2}), (\ref{B3}), (\ref{B4}) we have
\begin{eqnarray}
\bigg|\sum_{l=0}^{n-1}\varphi(\sigma^l\omega)-\alpha\bigg|\leq\frac{2}{q_n}+\frac{2}{n}\sum_{l=0}^{n}\var_{l}\varphi+(||g||_{\infty}+|\alpha|)\frac{2}{q_n}.
\end{eqnarray}
Note that since $\varphi$ is continuous we have $\frac{1}{n}\sum_{l=0}^{n}\var_{l}(\varphi) \rightarrow 0$. Thus, dividing by $n$ and letting $n\rightarrow \infty$ proves the lemma.
\end{proof}

\begin{lemma} \label{SBirkhoff} $S\subseteq \Sigma_{\alpha}^{\varphi}$.
\end{lemma}
\begin{proof} $S\subseteq \Sigma_{\alpha}$ is precisely Lemma \ref{Birkhoff Limits} (iii).
\end{proof}

\begin{lemma} \label{EntropyLemma} For all $\omega=(i_{\nu},j_{\nu})_{\nu \in \N}  \in S$ we have
\begin{enumerate}
\vspace{4mm}
\item [(i)] $ \lim_{n\rightarrow \infty}\frac{1}{n} \log \W([\omega_1 \cdots \omega_{n}])=-h(\mu,\sigma)$
\vspace{4mm}
\item [(ii)]  $ \lim_{n\rightarrow \infty}\frac{1}{n} \log \W([i_1 \cdots i_{n}])=-h^v(\mu,\sigma)$.
\end{enumerate}
\end{lemma}
\begin{proof} Both proofs resemble that of Lemma \ref{Birkhoff Limits}. We prove only part $(i)$ since the proof of part $(ii)$ is similar.\newline
Take $\omega \in S$. Given $n\in \N$ we choose $q_n$ so that $\gamma_{q_n}\leq n$ is maximal. Since $\gamma_{q_n}-\gamma_{q_n-1}\geq N(q_n)$ and $\gamma_{q_n}-\gamma_{q_n-1} \leq n$ we have
\begin{eqnarray} \label{ENT1}
\big| \log \mu_{k(q_n)}([\omega_{\gamma_{q_n-1}+1} \cdots \omega_{\gamma_{q_n}}])+(\gamma_{q_n}-\gamma_{q_n-1})h(\mu,\sigma)\big|&<& \frac{n}{q_n}
\end{eqnarray}
by Lemma \ref{ q lemma} (iv). Moreover, by the construction of $(\gamma_q)_{q \in \N}$,
\begin{equation}
\max\{B(l): l\leq q_n\} \gamma_{q_n-1}\leq \gamma_{q_n}/q_n \leq n/q_n
\end{equation}
 so by Lemma \ref{ q lemma} (i)
\begin{eqnarray} \label{ENT2}
\bigg|\sum_{l=1}^{q_n-1}\log \mu_{k(l)}([\omega_{\gamma_{l-1}+1}\cdots
\omega_{\gamma_l}])+\gamma_{q_n-1}h(\mu,\sigma)\bigg|&<&\frac{n}{q}.\end{eqnarray} Now either $n-\gamma_{q_n}\geq
N(q_n+1)$ or $n -\gamma_{q_n}\leq N(q_n+1)$. In the former case we apply Lemma \ref{ q lemma} (iv) and note that
$n- \gamma_{q_n}\leq n$ to obtain
\begin{eqnarray} \label{ENT3}
\bigg| \log \mu_{k(q_n+1)}([\omega_{\gamma_{q_n}} \cdots \omega_{n}])+(n-\gamma_{q_n})h(\mu,\sigma)\bigg|&<& \frac{n}{q_n+1}.
\end{eqnarray}
In the latter case, by the construction of $(\gamma_q)_{q \in \N}$ we have $N(q_n+1) \leq \gamma_{q_n}/q_n \leq n/q_n$ and hence
\begin{eqnarray} \label{ENT4}
\bigg| \log \mu_{k(q_n+1)}([\omega_{\gamma_{q_n}} \cdots \omega_{n}])+(n-\gamma_{q_n})h(\mu,\sigma)\bigg|&<& \frac{n}{q_n}.
\end{eqnarray}
Thus, by (\ref{ENT1}), (\ref{ENT2}), (\ref{ENT3}), (\ref{ENT4}) together with the construction of $\W$ we have
\begin{eqnarray}
\bigg| \log \W([\omega_{1} \cdots \omega_{n}])+nh(\mu,\sigma)\bigg|&<& \frac{3n}{q_n}.
\end{eqnarray}
Dividing by $n$ and letting $n\rightarrow \infty$ proves the lemma.
\end{proof}

\begin{lemma}\label{R convergence}
For all $\omega \in S$ we have $\lim_{n\rightarrow \infty}\frac{R_n(\omega)}{n}=0$.
\end{lemma}
\begin{proof}
This follows from Lemma \ref{ q lemma} (iii) and (ix) in a similar way to the proof of Lemma \ref{EntropyLemma}.
\end{proof}

\begin{lemma} \label{Symbolic Pointwise Dim W}
For all $\omega\in S$
\[\liminf_{n\rightarrow \infty} \frac{\log \W(B_n(\omega))}{\log \prod_{\nu=1}^n b_{i_{\nu}}}\geq D_{LY}(\mu).\]
\end{lemma}
\begin{proof}
Fix $\omega=(i_{\nu},j_{\nu})_{\nu \in \N} \in S$. By (\ref{Lyapunov Ratios}) we have
\begin{equation}
\lim_{n \rightarrow \infty} \frac{1}{n} \log \prod_{\nu=1}^{L_n(\omega)} a_{\omega_{\nu}}= \lim_{n \rightarrow \infty} \frac{1}{n} \log \prod_{\nu=1}^n b_{i_{\nu}}.
\end{equation}
Hence, by Lemma \ref{Birkhoff Limits} (i) and (ii) we have
\begin{equation}\label{L}
\lim_{n \rightarrow \infty} \frac{L_n(\omega)}{n} = \frac{\lambda^v(\mu,\sigma)}{\lambda(\mu,\sigma)}
\end{equation}
Given $n\in\N$ let $q_n$ be the greatest integer satisfying $\gamma_{q_n-1}<L_n(\omega)$ and let $L^+_n(\omega):=\min\left\lbrace l \geq L_n(\omega): k(q_n)|(l-\gamma_{q_n-1})\right\rbrace$. By the construction of $(\gamma_q)_{q \in \N}$ we have $k(q_n)\leq \gamma_{q_n-1}/q_n\leq L_n(\omega)/q_n\leq  n/q_n$ and so by (\ref{L})
\begin{equation}\label{L+}
\lim_{n \rightarrow \infty} \frac{L^+_n(\omega)}{n} =
\lim_{n \rightarrow \infty} \frac{L_n(\omega) + o(k(q_n))}{n} =
\lim_{n \rightarrow \infty} \frac{L_n(\omega)}{n} = \frac{\lambda^v(\mu,\sigma)}{\lambda(\mu,\sigma)}.
\end{equation}
Moreover, $L^+_n(\omega) \geq L_n(\omega)$ so
\begin{equation}
B_n(\omega) \subseteq [\omega_1\cdots\omega_{L^+_n(\omega)}]\cap \sigma^{-L^+_n(\omega)}[i_{L^+_n(\omega)}\cdots i_n].
\end{equation}
Thus, by Lemma \ref{Birkhoff Limits} (ii) it suffices to show that
\begin{equation}
\lim_{n\rightarrow \infty} \frac{1}{n}\W([\omega_1\cdots\omega_{L^+_n(\omega)}]\cap \sigma^{-L^+_n(\omega)}[i_{L^+_n(\omega)}\cdots i_n])= \frac{\lambda^v(\mu,\sigma)}{\lambda(\mu,\sigma)}h(\mu,\sigma) +\left( 1-\frac{\lambda^v(\mu,\sigma)}{\lambda(\mu,\sigma)}\right)h^v(\mu,\sigma).
\end{equation}
Now since $L^+_n(\omega) -\gamma_{q_n-1}$ is a multiple of $k(q_n)$ it follows from the construction of $\W$ that for all $\tau=(\tau_{\nu})_{\nu=1}^n\in \D^n$ we have
\begin{equation}\label{independence}
\W([\tau_1 \cdots \tau_n])=\W([\tau_1\cdots \tau_{L^+_n(\omega)}])\W([ \tau_{L^+_n(\omega)+1}\cdots \tau_n]).
\end{equation}
Hence, it suffices to show that
\begin{eqnarray}
\label{Entlim1}\lim_{n\rightarrow \infty}\frac{1}{n} \W([\omega_1\cdots\omega_{L^+_n(\omega)}])&=& \frac{\lambda^v(\mu,\sigma)}{\lambda(\mu,\sigma)}h(\mu,\sigma)\\
\label{Entlim2} \lim_{n\rightarrow \infty} \frac{1}{n} \W(\sigma^{-L^+_n(\omega)}[i_{L^+_n(\omega)}\cdots i_n])&=& \left( 1-\frac{\lambda^v(\mu,\sigma)}{\lambda(\mu,\sigma)}\right)h^v(\mu,\sigma).
\end{eqnarray}
Equation (\ref{Entlim1}) follows from Lemma \ref{EntropyLemma} (i) combined with (\ref{L+}). Equation
(\ref{Entlim1}) follows from Lemma \ref{EntropyLemma} (ii) and (\ref{L+}) along with
\begin{equation}
\W([i_1 \cdots i_n])=\W([i_1\cdots i_{L^+_n(\omega)}])\W([i_{L^+_n(\omega)+1}\cdots i_n])
\end{equation}
which follows from (\ref{independence}).
\end{proof}
\begin{lemma}\label{LB}$\dim(\Pi(\Sigma_{\alpha}^{\varphi})) \geq D_{LY}(\mu)$.
\end{lemma}
\begin{proof}
We begin with the special cases in which $\D$ is not two dimensional. If $\D$ is not two dimensional then either
there is just one $i$ for which there exists $j$ with $(i,j)\in\D$, or for each $i$ there is just one $j$ for
which $(i,j) \in \D$. In the former case we have $h^v(\mu,\sigma)=0$. Thus it suffices to show that
\begin{equation}
\dim(\Pi(\Sigma_{\alpha}^{\varphi})) \geq \frac{h(\mu,\sigma)}{\lambda(\mu,\sigma)}.
\end{equation}
Moreover this follows from Olsen \cite{Olsen Multifractal 1} Theorem 1 applied to the projection
$\chi^h(\Lambda_{\alpha}^{\varphi})$ onto the horizontal axis, together with the fact that the Hausdorff dimension
cannot increase under projection. Similarly when there is just one $j$ for each $i$ we have
$h(\mu,\sigma)=h^v(\mu,\sigma)$ and so it suffices to show
\begin{equation}
\dim(\Pi(\Sigma_{\alpha}^{\varphi})) \geq \frac{h^v(\mu,\sigma)}{\lambda^v(\mu,\sigma)}
\end{equation}
which follows from Olsen \cite{Olsen Multifractal 1} Theorem 1 applied to the projection
$\chi^v(\Lambda_{\alpha}^{\varphi})$ onto the vertical. Henceforth we assume that $\D$ is two dimensional.
\newline
By Lemma \ref{SBirkhoff} it suffices to prove that $\dim(\Pi(S)) \geq D_{LY}(\mu)$. Now by Lemma \ref{S>0} we have
$\W(S)>0$ and by Lemma \ref{Symbolic Pointwise Dim W} we have
\begin{equation}
\liminf_{n\rightarrow \infty} \frac{\log \W(B_n(\omega))}{\log \prod_{\nu=1}^n b_{i_{\nu}}}\geq D_{LY}(\mu)
\end{equation}
for all $\omega \in S$. Moreover, by Lemma \ref{R convergence} for all $\omega \in S$
\begin{equation}
\lim_{n\rightarrow\infty}\frac{R_n(\omega)}{n}=0.
\end{equation}
Thus, by Lemma \ref{SymbolicDimensionLemmas} (i), combined with the assumption that $\D$ is two dimensional, the
lemma holds.
\end{proof}
Lemma \ref{LB} holds for all $\mu \in \M_{\sigma}(\Sigma)$ satisfying $\int \varphi d\mu=\alpha$. Therefore,
\begin{equation}
\dim{\Lambda^{\varphi}_{\alpha}}\geq \sup \left\lbrace D_{LY}(\mu): \mu \in \M_{\sigma}(\Sigma), \int \varphi d\mu =\alpha \right\rbrace.
\end{equation}

\section{Proof of the upper bound}\label{Proof of the upper bound}
We begin by demonstrating that the function $f:[\alpha_{\min}(\varphi), \alpha_{\max}(\varphi)] \rightarrow \R$ given by
\begin{equation}
f(\alpha)=\sup \left\lbrace D_{LY}(\mu): \mu \in \M_{\sigma}(\Sigma), \int \varphi d\mu =\alpha \right\rbrace
\end{equation}
is continuous on $[\alpha_{\min}(\varphi), \alpha_{\max}(\varphi)]$.

\begin{lemma}\label{continuity}
$f$ is continuous on $[\alpha_{\min}(\varphi), \alpha_{\max}(\varphi)] $.
\end{lemma}
\begin{proof}
Fix $\alpha_* \in [\alpha_{\min}(\varphi), \alpha_{\max}(\varphi)]$. First we show that $f$ is upper semi-continuous at $\alpha^*$. Since $[\alpha_{\min}(\varphi), \alpha_{\max}(\varphi)]$ $=\left\lbrace \alpha \in \R: \Sigma_{\alpha}^{\varphi}\neq \emptyset \right\rbrace$ we may take a sequence $\{\alpha_n\}_{n\in\N}$ such that $\lim_{n\rightarrow \infty}\alpha_n= \alpha$, $\lim_{n\rightarrow \infty}f(\alpha_n)=\limsup_{\alpha\rightarrow \alpha_*} f(\alpha)$ and for each $n$ there is a measure $\mu_n$ such that $\int \varphi d \mu_n=\alpha_n$ and $D_{LY}(\mu_n)>f(\alpha_n)-1/n$. Since $\M_{\sigma}(\Sigma)$ is compact we may take a weak $*$ limit $\mu_*$ of $\{\mu_n\}_{n\in\N}$. It follows from the upper semi-continuity of entropy (see \cite{Walters} Theorem 8.2) that $\mu \mapsto D_{LY}(\mu)$ is upper semi-continuous. Thus, $f(\alpha_*)\geq D_{LY}(\mu_*)\geq \limsup_{n \rightarrow \infty}D_{LY}(\mu_n)\geq \limsup_{\alpha\rightarrow \alpha_*} f(\alpha)$. Hence, $f$ is upper-semicontinuous at $\alpha^*$ for all $\alpha_* \in [\alpha_{\min}(\varphi), \alpha_{\max}(\varphi)]$.\newline
To prove that $f$ is lower semi-continuous we first show that, provided $\alpha_* \neq \alpha_{\max}(\varphi)$, $\liminf^{\alpha>\alpha_*}_{\alpha\rightarrow \infty} f(\alpha)> f(\alpha_*)-\epsilon$ for an arbitrary $\epsilon>0$. Choose $\mu_{\epsilon}\in \M_{\sigma}(\Sigma)$ with $\int \varphi d \mu_{\epsilon}= \alpha_*$ and $D_{LY}(\mu_{\epsilon})>f(\alpha_*)-\epsilon$. Take $\mu_{\max} \in \M_{\sigma}(\Sigma)$ with $\int \varphi d \mu_{\max}= \alpha_{\max}$. Now for each $\rho\in(0,1)$ we let $\mu_{\rho, \epsilon}:=\rho \mu_{\epsilon}+(1-\rho)\mu_{\max}$. Note that $\int f d\mu_{\rho,\epsilon}=\rho\alpha_{*}+(1-\rho)\alpha_{\max}$. Moreover, it follows from the fact that the entropy map is affine (see \cite{Walters} Theorem 8.1) that $\rho\mapsto D_{LY}(\mu_{\rho,\epsilon})$ is continuous.  Hence $f(\alpha_*)-\epsilon< D_{LY}(\mu_*)$ $\leq \liminf_{\rho\rightarrow 1} D_{LY}(\mu_{\rho,\epsilon})$ $ \leq \liminf_{\rho\rightarrow 1} f(\rho\alpha_*+(1-\rho)\alpha_+)=  \liminf^{\alpha>\alpha_*}_{\alpha\rightarrow \infty} f(\alpha)$. The proof that $\liminf^{\alpha<\alpha_*}_{\alpha\rightarrow \infty} f(\alpha)\geq f(\alpha_*)$ for all $\alpha_* \neq \alpha_{\min}(\varphi)$ is similar. Thus, $f$ is lower semi-continuous at $\alpha^*$ for all $\alpha_* \in [\alpha_{\min}(\varphi), \alpha_{\max}(\varphi)]$.
\end{proof}
Thus, to show that the spectrum is continuous it suffices to identify the prove that $\dim
\Lambda_{\alpha}^{\varphi} = f(\alpha)$.\newline Another consequence of Lemma \ref{continuity} is that in proving
the upper bound in Theorem \ref{main}, $\dim \Lambda_{\alpha}^{\varphi}\leq f(\alpha)$, it suffices to prove
\begin{equation}\label{e estimate}
\dim \Lambda^{\varphi}_{\alpha} \leq \sup \left\lbrace D_{LY}(\mu): \mu \in \M_{\sigma}(\Sigma), \big|\int \varphi d\mu -\alpha\big|\leq \epsilon \right\rbrace
\end{equation} for arbitrarily small $\epsilon>0$.\newline
The key lemma in the proof of the upper bound is Lemma \ref{Key Lemma for the Upper Bound}, which uses an idea
from Bara\'{n}ski \cite{Baranski} to give an upper estimate for the dimension of the projection of a subset of the
symbolic space in terms the possible limit points for frequencies of words amongst its members. From Lemma
\ref{Key Lemma for the Upper Bound} we can deduce an estimate of the form (\ref{e estimate}) with an error term
given by the variance of a potential across sets of strings with a common first digit $\var_1(\varphi)$ (see Lemma
\ref{upper estimate with var 1 error}). By iterating our system some large number of times we are able to
transform this estimate into estimates of the form (\ref{e estimate}) with an arbitrary degree of precision (see
Lemma \ref{upper estimate with var k error}).\newline We introduce the following terminology for the proof of the
upper bound. Define
\begin{equation}
\Prob:=\left\lbrace (p_d)_{d\in\D}\in [0,1]^{\D}: \sum_{d\in\D}p_d=1\right\rbrace
\end{equation}
 be the simplex of probability vectors on the digit set $\D$ and
 \begin{equation}
 \mathbb{B}:=\left\lbrace (p_d)_{d\in\D}\in \Prob: p_d \in \Q \backslash \{0\}\right\rbrace.
\end{equation}
Note that $\Prob$ is compact and $\mathbb{B}$ is a countable dense subset. For each $\p\in \Prob$ we let $\mu_{\p}$ denote the corresponding Bernoulli measure on $\Sigma$. Given $(i,j)\in\D$ we define, for each $\omega \in \Sigma$ and $n\in \N$
\begin{equation}
N_{ij}(\omega|n):= \#\{l\in \{1, \cdots,n\}:\omega_l=(i,j)\}
\end{equation}
and $P_{ij}(\omega|n):=N_{ij}(\omega|n)/n$. This implies that for each $\omega\in\Sigma$ and $n\in\N$ we have a probability vector $\mathbf{P}(\omega|n):=(P_{ij}(\omega|n))_{(i,j)\in\D}\in\Prob$ known as the nth level frequency vector for $\omega$. We also let $N_i(\omega|n):=\sum_{j=1}^{M_i}N_{ij}(\omega|n)$, $P_i(\omega|n):=\sum_{j=1}^{M_i}P_{ij}(\omega|n)$ and $p_i:=\sum_{j=1}^{M_i}p_{ij}$ for $(p_{ij})_{(i,j)\in\D}\in\Prob$.
\begin{lemma}\label{Key Lemma for the Upper Bound}
Suppose we have $\Omega \subseteq \Sigma$ and $A \subseteq \Prob$ such that for all $\omega \in \Omega$ every limit point of the sequence $(\mathbf{P}(\omega|n))_{n\in\N}$ of frequency vectors for $\omega$ lies within $A$. Then $\dim{\Pi(\Omega)}\leq \sup \left\lbrace D_{LY}(\mu_p) : p \in A \right\rbrace$.
\end{lemma}
\begin{proof}
Recall that $D_{LY}(\mu):= \frac{h(\mu,\sigma)}{\lambda(\mu,\sigma)}
+\bigg(\frac{1}{\lambda^v(\mu,\sigma)}-\frac{1}{\lambda(\mu,\sigma)}\bigg) h^v(\mu,\sigma)$, so for Bernoulli
measures $\mu_{\p}$ we can write
\begin{eqnarray*}
D_{LY}(\mu_{\p})&=&\frac{\sum_{(i,j) \in \D} p_{ij} \log p_{ij}}{\sum_{(i,j) \in \D} p_{ij} \log a_{ij}}+\frac{\sum_{i=1}^N p_{i} \log p_{i}}{\sum_{i=1}^N p_{i}\log b_{i}} - \frac{\sum_{i=1}^N p_{i} \log p_{i}}{\sum_{(i,j) \in \D} p_{ij} \log a_{ij}}\\
&=& \frac{\sum_{(i,j) \in \D} p_{ij} \log p_{ij}/p_{i} }{\sum_{(i,j) \in \D} p_{ij} \log
a_{ij}}+\frac{\sum_{i=1}^N p_{i} \log p_{i}}{\sum_{i=1}^N p_{i}\log b_{i}}.
\end{eqnarray*}
Let $s:= \sup\{D_{LY}(\mu_p): p \in A\}$. Fix some $\delta>0$ and $\omega \in \Omega$. For each $n\in \N$ we take $\mathbf{\rho}(n)=(\rho_{ij}(n))_{(i,j)\in\D} \in \Prob$ defined
by
\begin{eqnarray}\label{rho def}
\rho_{ij}(n):=\begin{cases} P_{i}(\omega|n) \frac{
P_{ij}(\omega|L_n(\omega))}{P_{i}(\omega|L_n(\omega))} \text{  if  }P_{ij}(\omega|L_n(\omega))\neq 0\\0\text{ if  }P_{ij}(\omega|L_n(\omega))=0\end{cases}.
\end{eqnarray}
Since $\Prob$ is compact we may take $n_q$ such that:
\begin{enumerate}
\vspace{2mm} \item [(i)] $\lim_{q \rightarrow \infty} \frac{\sum_{i=1}^{N} P_{i}(\omega|n_q) \log
P_{i}(\omega|n_q) }{\sum_{i=1}^{N} P_{i}(\omega|n_q) \log  b_{i}}$ $ = \liminf_{n \rightarrow \infty}
\frac{\sum_{i=1}^{N} P_{i}(\omega|n) \log  P_{i}(\omega|n) }{\sum_{i=1}^{N} P_{i}(\omega|n) \log  b_{i}}$;
\vspace{2mm} \item [(ii)]$\mathbf{\rho}:=\lim_{q \rightarrow \infty} \rho(n_q)$ exists; \vspace{2mm} \item [(iii)]
$\mathbf{P}:=\lim_{q \rightarrow \infty} \mathbf{P}(\omega|L_{n_q}(\omega))$ exists.
 \end{enumerate}
\vspace{2mm}
Since $\omega \in \Omega$, $\mathbf{P} \in A$. Letting $b_{\max}:=\max_i{b_i}<1$ we may take $\mathbf{\beta}(\omega)=({\beta^{\omega}}_{ij})_{(i,j)\in \D} \in \mathbb{B}$ such that ${\rho}_{ij}/\beta^{\omega}_{ij}, \beta^{\omega}_{ij}/{\rho}_{ij} <
b_{\max}^{-\delta/2}$ for all $(i,j) \in \D$ with $\rho_{ij} \neq 0$. By the definition of $B_{n_q}(\omega)$ we have
\begin{equation} \mu_{\mathbf{\beta({\omega})}}(B_{n_q}(\omega))= \prod_{\nu=1}^{L_{n_q}(\omega)}\beta^{\omega}_{i_{\nu}j_{\nu}} \times \prod_{\nu=L_{n_q}(\omega)+1}^{n_q}\beta^{\omega}_{i_{\nu}}
\end{equation}
and so
\begin{eqnarray*}
\log\mu_{\mathbf{\beta^{(\omega)}}}(B_{n_q}(\omega))&=& \sum_{(i,j) \in \D} N_{ij}(\omega|L_{n_q}(\omega)) \log \beta^{\omega}_{ij} +  \sum_{i=1}^N(N_{i}(\omega|n_q)-N_{i}(\omega|L_{n_q}(\omega))) \log \beta^{\omega}_{i}\\
 &=& \sum_{(i,j) \in \D} N_{ij}(\omega|L_{n_q}(\omega)) \log \beta^{\omega}_{ij}/\beta^{\omega}_{i} +  \sum_{i=1}^N N_{i}(\omega|n_q) \log \beta^{\omega}_{i}.\\
\end{eqnarray*}

By (\ref{Lyapunov Ratios}) we have
\begin{equation}
\log a_{\min} \leq \sum_{i} N_{i}(\omega|n_q) \log b_{i}-\sum_{ij}N_{ij}(\omega|L_{n_q}(\omega)) \log a_{ij}\leq 0.\end{equation}
Hence, \begin{eqnarray*}
\liminf_{q \rightarrow \infty} \frac{ \mu_{\beta(\omega)}(B_{n_q}(\omega))}{\log
{\prod_{\nu=1}^{n_q} b_{i_{\nu}}}}&\leq& \liminf_{q \rightarrow \infty} \frac{\sum_{ij}N_{ij}(\omega|L_{n_q}(\omega)) \log
\beta^{\omega}_{ij}/\beta^{\omega}_{i}}{\sum_{ij}N_{ij}(\omega|L_{n_q}(\omega)) \log
a_{ij}}+\frac{\sum_{i} N_{i}(\omega|n_q) \log  \beta^{\omega}_i}{\sum_{i}
N_{i}(\omega|n_q) \log b_{i}}\\
&\leq& \liminf_{q \rightarrow \infty} \frac{\sum_{ij}P_{ij}(\omega|L_{n_q}(\omega)) \log
\beta^{\omega}_{ij}/\beta^{\omega}_{i}}{\sum_{ij}P_{ij}(\omega|L_{n_q}(\omega)) \log
a_{ij}}+\frac{\sum_{i} P_{i}(\omega|n_q) \log  \beta^{\omega}_i}{\sum_{i}
P_{i}(\omega|n_q) \log b_{i}}.\\
\end{eqnarray*}

Since $\mathbf{\beta(\omega)}$ ${\rho}_{ij}/\beta^{\omega}_{ij}, \beta^{\omega}_{ij}/{\rho}_{ij} <
b_{\max}^{-\delta/2}$ for all $(i,j) \in \D$ such that $\rho_{ij} \neq 0$ we have
\begin{eqnarray*}
\liminf_{q \rightarrow \infty} \frac{\sum_{ij}P_{ij}(\omega|L_{n_q}(\omega)) \log
\beta^{\omega}_{ij}/\beta^{\omega}_{i}}{\sum_{ij}P_{ij}(\omega|L_{n_q}(\omega)) \log
a_{ij}}+\frac{\sum_{i} P_{i}(\omega|n_q) \log  \beta^{\omega}_i}{\sum_{i}
P_{i}(\omega|n_q) \log b_{i}}\\ \leq \liminf_{q \rightarrow \infty} \frac{\sum_{ij}P_{ij}(\omega|L_{n_q}(\omega)) \log
\rho_{ij}(n_q)/\rho_{i}(n_q)}{\sum_{ij}P_{ij}(\omega|L_{n_q}(\omega)) \log
a_{ij}}+\frac{\sum_{i} P_{i}(\omega|n_q) \log  \rho_{i}(n_q)}{\sum_{i} P_{i}(\omega|n_q) \log b_{i}}+ \delta.
\end{eqnarray*}
By the definition of $\mathbf{\rho}(n_q)$ (\ref{rho def})
\begin{eqnarray*}
\liminf_{q \rightarrow \infty} \frac{\sum_{ij}
P_{ij}(\omega|L_{n_q}(\omega)) \log
\rho_{ij}(n_q)/\rho_{i}(n_q)}{\sum_{ij}P_{ij}(\omega|L_{n_q}(\omega))
\log a_{ij}}
+\frac{\sum_{i} P_{i}(\omega|n_k) \log
\rho_{i}(n_q)}{\sum_{i} P_{i}(\omega|n_q) \log b_i}\\
= \liminf_{q
\rightarrow \infty} \frac{\sum_{ij} P_{ij}(\omega|L_{n_q}(\omega))
\log P_{ij}(\omega|L_{n_q}(\omega))/P_{i}(\omega|L_{n_q}(\omega))
}{\sum_{ij}P_{ij}(\omega|L_{n_q}(\omega)) \log
a_{ij}}
+\frac{\sum_i P_{i}(\omega|n_q) \log  P_{i}(\omega|n_q)
}{\sum_{i} P_i(\omega|n_q) \log b_{i}}.
\end{eqnarray*}

By the first condition on $n_q$ we have
\begin{eqnarray*}
\liminf_{q \rightarrow \infty} \frac{\sum_{ij} P_{ij}(\omega|L_{n_q}(\omega))
\log P_{ij}(\omega|L_{n_q}(\omega))/P_{i}(\omega|L_{n_q}(\omega))
}{\sum_{ij}P_{ij}(\omega|L_{n_q}(\omega)) \log
a_{ij}}\\ +\frac{\sum_{i}P_{i}(\omega|n_q) \log  P_{i}(\omega|n_q)
}{\sum_{i} P_i(\omega|n_q) \log b_{i}}\\
\leq \liminf_{q \rightarrow \infty} \frac{\sum_{ij} P_{ij}(\omega|L_{n_q}(\omega))
\log P_{ij}(\omega|L_{n_q}(\omega))/P_{i}(\omega|L_{n_q}(\omega))
}{\sum_{ij} P_{ij}(\omega|L_{n_q}(\omega)) \log
a_{ij}}\\
+\frac{\sum_{i} P_{i}(\omega|L_{n_q}(\omega)) \log  P_{i}(\omega|L_{n_q}(\omega))
}{\sum_{i} P_i(\omega|L_{n_q}(\omega)) \log b_{i}}.
\end{eqnarray*}
Since $\lim_{q \rightarrow \infty} \mathbf{P}(\omega|L_{n_q}(\omega))=\mathbf{P}$ and $\mathbf{P} \in A$ we have
\begin{eqnarray*}
\lim_{q \rightarrow \infty} \frac{\sum_{ij} P_{ij}(\omega|L_{n_q}(\omega))
\log P_{ij}(\omega|L_{n_q}(\omega))/P_{i}(\omega|L_{n_q}(\omega))
}{\sum_{ij}P_{ij}(\omega|L_{n_q}(\omega)) \log
a_{ij}}\\
+\frac{\sum_{i} P_{i}(\omega|L_{n_q}(\omega)) \log  P_{i}(\omega|L_{n_q}(\omega))
}{\sum_{i} P_i(\omega|L_{n_q}(\omega)) \log b_{i}}\\
= \frac{\sum_{ij} P_{ij} \log P_{ij}/P_{i} }{\sum_{ij}P_{ij} \log a_{ij}}+\frac{\sum_{i} P_{i} \log P_{i}}{\sum_{i}P_{ij}\log b_{i}} \leq s.
\end{eqnarray*}
Hence, for each $\omega \in \Omega$ we may find $\mathbf{\beta(\omega)} \in \mathbb{B}$ such that
\begin{equation} \label{PDimUB}\liminf_{n\rightarrow \infty}\frac{\log \mu_{\mathbf{\beta(\omega)}}( B_n(\omega))}{\log \prod_{\nu=1}^n b_{i_{\nu}j_{\nu}}} \leq s + \delta.
\end{equation}
Letting $\Lambda^{\Omega}(\beta):= \{x \in \Pi(\Omega): \beta(x)=\beta \}$ for each $\beta \in \mathbb{B}$ we have $\Pi(\Omega)=\bigcup_{\beta\in\mathbb{B}} \Lambda^{\Omega}(\beta)$. Moreover, by (\ref{PDimUB}) combined with Lemma \ref{SymbolicDimensionLemmas} (ii) we have $\dim \Lambda^{\Omega}(\beta)\leq s+\delta$ for each $\beta \in \mathbb{B}$. Since $\dim$ is closed under countable unions it follows that
$\dim \Pi(\Omega)\leq s+\delta$. Letting $\delta \rightarrow 0$ proves the lemma.
\end{proof}

We now make a quick digression to see how Lemma \ref{Key Lemma for the Upper Bound} implies the following generalization of a result due to Nielsen \cite{Nielsen}. Given $\p=(p_{ij})_{(i,j)\in\D}\in\Prob$ we define
\begin{equation}
\Lambda(\p):= \left\lbrace x=\Pi(\omega)\in \Lambda: \lim_{n\rightarrow \infty} P_{ij}(\omega|n)=p_{ij} \text{ for all } (i,j) \in \D \right\rbrace.
\end{equation}

\begin{corollary} \label{Nielsen} For each $\p\in\Prob$ $\dim(\Lambda(\p))= D_{LY}(\mu_{\p})$.
\end{corollary}
\begin{proof} The lower bound follows from several applications of the Kolmogorov's strong law of large numbers combined with Lemma \ref{SymbolicDimensionLemmas} (ii). The upper bound is an immediate consequence of Lemma \ref{Key Lemma for the Upper Bound} with $A=\{\p\}$.
\end{proof}
Returning to the proof of Theorem \ref{main} we obtain our first upper estimate for $\dim \Lambda_{\alpha}^{\varphi}$.
\begin{lemma}\label{upper estimate with var 1 error}
\begin{equation*}
\dim \Lambda^{\varphi}_{\alpha} \leq\sup \left\lbrace D_{LY}(\mu): \mu \in \B_{\sigma}(\Sigma), \big|\int \varphi d\mu -\alpha\big|\leq \var_1(\varphi) \right\rbrace.
\end{equation*}
\end{lemma}
\begin{proof}
By Lemma \ref{Key Lemma for the Upper Bound} it suffices to show that given $\omega \in \Sigma_{\alpha}^{\varphi}$ and $\p=(p_{ij})\in \Prob$ a limit point for the sequence $(\mathbf{P}(\omega|n))_{n\in \N}$ we have $\big|\int \varphi d\mu_{\p} -\alpha\big|\leq \var_1(\varphi)$. Now given $(i,j)\in \D$ we have $\big|\int_{[(i,j)]} \varphi d\mu_{\p}-\varphi(\tau)\big|\leq \var_1(\varphi)$ for all $\tau\in \Sigma$ with $\tau_1=(i,j)$. Thus, for all $n\in \N$ we have
\begin{equation}
\bigg|\sum_{(i,j) \in \D} P_{ij}(\omega|n)\int_{[(i,j)]} \varphi d\mu_{\p}-\frac{1}{n} \sum_{l=0}^{n-1}\varphi(\omega^l)\bigg|\leq \var_1(\varphi).
\end{equation}
Since $\p$ is a limit point of $(\mathbf{P}(\omega|n))_{n\in \N}$ and $\omega \in  \Sigma_{\alpha}^{\varphi}$ this implies
\begin{equation}
\bigg|\sum_{(i,j) \in \D} p_{ij}\int_{[(i,j)]} \varphi d\mu_{\p}-\alpha\bigg|\leq \var_1(\varphi).
\end{equation}
 Since $\int \varphi d\mu_{\p}= \sum_{(i,j) \in \D} p_{ij}\int_{[(i,j)]} \varphi d\mu_{\p}$ this completes the proof of the lemma.
\end{proof}

Lemma \ref{upper estimate with var 1 error} proves the special case of Theorem \ref{main} for which $\var_1(\varphi)=0$. To prove the upper bound in Theorem \ref{main} in full generality requires a little more work. We iterate our Lalley-Gatzouras system many times to form new Lalley-Gatzouras systems to which we apply Lemma \ref{upper estimate with var 1 error} to obtain increasingly precise estimates for the upper bound. Take $k\in\N$. For each finite string $\xi:=\xi_1 \cdots \xi_k \in\D^k$ we let
\begin{equation}
S_{\xi}:=S_{\xi_1} \circ \cdots \circ S_{\xi_k}.
\end{equation}
It follows from the fact that $(S_{ij})_{(i,j) \in\D}$ is a Lalley-Gatzouras system that $(S_{\xi})_{\xi \in\D^k}$ is also a Lalley-Gatzouras system, which we call the $k$-th level Lalley-Gatzouras system. $\Sigma$ may be identified with the full shift $(\D^k)^{\N}$. The corresponding left shift is then just $k$ times the ordinary left shift, $\sigma^k$. Thus, in order to relate the $k$-th level Lalley-Gatzouras system back to our original Lalley-Gatzouras system we will require a lemma relating members of $\M_{\sigma^{k}}(\Sigma)$ to members of $\M_{\sigma}(\Sigma)$. Define a potential $A_k (\varphi):\Sigma \rightarrow \R$ by $A_k (\varphi):=\frac{1}{k}\sum_{l=0}^{k-1} \varphi\circ \sigma^l$ and for each $\nu \in \M_{\sigma^k}(\Sigma)$ define a Borel probability measure $A_k(\nu)$ by $A_k(\nu):=\frac{1}{k}\sum_{l=0}^{k-1} \nu \circ \sigma^{-l}$. Similarly if $\nu \in \M_{\sigma_v^k}(\Sigma_v)$ we let $A^v_k(\nu):=\frac{1}{k}\sum_{l=0}^{k-1} \nu \circ \sigma_v^{-l}$.

\begin{lemma}\label{averagemeasure}
Take $\nu \in \M_{\sigma^k}(\Sigma)$ and let $\mu=A_k(\nu)$. Then,
\begin{enumerate}
\vspace{2mm}
\item[(i)] $\mu\in\M_{\sigma}(\Sigma)$
\vspace{4mm}
\item[(ii)] $h(\mu,\sigma)=\frac{1}{k} h(\nu,\sigma^k)$
\vspace{4mm}
\item[(iii)] $h^v(\mu,\sigma)=\frac{1}{k}h^v(\nu,\sigma^k)$
\vspace{4mm}
\item[(iv)] $\lambda(\mu,\sigma)=\frac{1}{k} \lambda(\nu,\sigma^k)$
\vspace{4mm}
\item[(v)]  $\lambda^v(\mu,\sigma)=\frac{1}{k}\lambda^v(\nu,\sigma^k)$
\vspace{4mm}
\item[(vi)] $\int \varphi d\mu=\int A_k(\varphi)  d\nu$
\vspace{4mm}
\item[(vii)] $D_{LY}(\mu)= D_{LY}^k(\nu)$.
\end{enumerate}
\end{lemma}
\begin{proof}
Parts (i), (i), (iv), (v) and (vi) follow from \cite{non-uniformly hyperbolic} Lemma 2. Since $\pi \circ \sigma = \sigma_v \circ \pi$ we have $A^v_k(\pi(\nu))=\pi(A_k(\nu))$ and hence (iii) also follows from \cite{non-uniformly hyperbolic} Lemma 2. Part (vii) follows from parts (i), (ii), (iii), (iv) and (v).
\end{proof}

\begin{lemma}\label{upper estimate with var k error}
\begin{equation*}
\dim \Lambda^{\varphi}_{\alpha} \leq\sup \left\lbrace D_{LY}(\mu): \mu \in \B_{\sigma}(\Sigma), \big|\int \varphi d\mu -\alpha\big|\leq \var_k(A_k(\varphi)) \right\rbrace.
\end{equation*}
\end{lemma}
\begin{proof} First note that $\frac{1}{n}\sum_{l=0}^{n-1}A_k(\varphi)\circ(\sigma^k)^l=\frac{1}{nk}\sum_{l=0}^{nk-1}\varphi\circ\sigma^l$ and hence $\lim_{n\rightarrow \infty}\frac{1}{n}\sum_{l=0}^{n-1}A_k(\varphi)((\sigma^k)^l(\omega))=\alpha$ for all $\omega \in \Sigma^{\varphi}_{\alpha}$. Thus, by applying Lemma \ref{upper estimate with var 1 error} to our $k$-th level Lalley-Gatzouras
system and noting that words of length $k$ and $\sigma^k$ invariant measures in the original system correspond,
respectively, to digits and shift invariant measures in the $k$-th level Lalley-Gatzouras system we have
\begin{equation} \label{k eq}
\dim \Lambda^{\varphi}_{\alpha} \leq\sup \left\lbrace D^k_{LY}(\nu): \nu \in \B_{\sigma^k}(\Sigma), \big|\int A_k(\varphi) d\nu -\alpha\big|\leq \var_k(A_k(\varphi)) \right\rbrace.
\end{equation} Combining (\ref{k eq}) with Lemma \ref{averagemeasure} proves the lemma.
\end{proof}
Since $\varphi \in C(\Sigma)$ is continuous we have $\lim_{k\rightarrow \infty} \var_k(A_k(\varphi))=0$. Thus, by Lemma \ref{upper estimate with var k error}, we have shown
\begin{equation}
\dim \Lambda^{\varphi}_{\alpha} \leq \sup \left\lbrace D_{LY}(\mu): \mu \in \M_{\sigma}(\Sigma), \big|\int \varphi d\mu -\alpha\big|\leq \epsilon \right\rbrace
\end{equation} for arbitrarily small $\epsilon>0$ and so by Lemma \ref{continuity} the result follows.

\section{Remarks}
Following Olsen and Winter \cite{Olsen Multifractal 1}, \cite{Olsen Winter} one may consider more general types of level sets. Given $A \subseteq [\alpha_{\min}(\varphi), \alpha_{\max}(\varphi)]$ we let $\overline{\Sigma^{\varphi}_{A}}$ denote the set of $\omega \in \Sigma$ for which every accumulation point of the sequence $\left( A_n(\varphi)(\omega)\right)_{n \in \N}$ lies within $A$ and $\overline{\Lambda^{\varphi}_{A}}:=\Pi(\overline{\Sigma^{\varphi}_{A}})$ its projection by $\Pi$. Then, by essentially the same argument as above, we have
\begin{equation}
\dim{\overline{\Lambda^{\varphi}_{A}}}= \sup \left\lbrace D_{LY}(\mu): \mu \in \M_{\sigma}(\Sigma), \int \varphi d\mu\in A \right\rbrace.
\end{equation}
Now suppose $A \subseteq [\alpha_{\min}(\varphi), \alpha_{\max}(\varphi)]$ is a compact sub-interval. Let $\underline{\Sigma^{\varphi}_{A}}$ denote the set of $\omega \in \Sigma$ for which the set of accumulation points of the sequence $\left( A_n(\varphi)(\omega)\right)_{n \in \N}$ is equal to $A$ and $\underline{\Lambda^{\varphi}_{A}}:=\Pi(\underline{\Sigma^{\varphi}_{A}})$.
By employing the methods of \cite{Olsen Winter} Theorem 3.1, along with a few ideas from section \ref{Proof of the lower bound}, one can prove the following lower estimate
\begin{equation} \label{OW sets}
\dim{\underline{\Lambda^{\varphi}_{A}}}\geq \inf_{\alpha \in A}\sup \left\lbrace D_{LY}(\mu): \mu \in \M_{\sigma}(\Sigma), \int \varphi d\mu=\alpha \right\rbrace.
\end{equation}
In particular, the projection of the set of points $\omega \in \Sigma$ for which $\left( A_n(\varphi)(\omega)\right)_{n \in \N}$ does not converge has dimension $\dim \Lambda$. However, it seems very plausible that the lower bound given by (\ref{OW sets}) is not always optimal and it would be interesting to know what the exact value of $\dim{\underline{\Lambda^{\varphi}_{A}}}$ is.


\begin{thebibliography}{15}
\bibitem{Barral Feng}
J. Barral and D. Feng, \textit{Weighted thermodynamic formalism and applications}, (2009). arXiv:0909.4247v1.
\bibitem{Barral Mensi}
J. Barral and M. Mensi, \textit{Multifractal analysis of Birkhoff averages on `self-affine' symbolic spaces}. Nonlinearity 21 (2008), no. 10, 2409-2425.
\bibitem{Baranski}
K. Bara\'{n}ski, \textit{Hausdorff dimension of the limit sets of some planar geometric constructions} Adv. Math. 210.1 (2007), 391-415.
\bibitem{Barriera Saussol}
L. Barriera and B. Saussol, \textit{Variational principles and mixed multifractal spectra}. Trans. Amer. Math. Soc. 353 (2001), no. 10, 3919-3944.
\bibitem{Bedford}
T. Bedford, \textit{PhD Thesis: Crinkly curves, Markov partitions and box dimension of self-similar sets}. Ph.D. thesis, University of Warwick (1984).
\bibitem{Falconer Techniques}
K. Falconer, \textit{Techniques in Fractal Geometry}. John Wiley and Sons, Ltd., Chichester, (1997).
\bibitem{Fan Feng Wu}
A. Fan, D. Feng and J. Wu, \textit{Recurrence, dimension and entropy}.  J. London Math. Soc. (2) 64 (2001), no. 1, 229--244.
\bibitem{Gatzouras Lalley}
S. P. Lalley and D. Gatzouras, \textit{Hausdorff and Box Dimension
of certain Self-Affine Fractals}, Indiana Univ. Math. J. 41 (1992), 533.
\bibitem{Gelfert Rams}
K. Gelfert and M. Rams, \textit{The Lyapunov spectrum of some parabolic systems}. Ergodic Theory Dynam. Systems 29 (2009), no. 3, 919--940.
\bibitem{non-uniformly hyperbolic}
A. Johansson, T. Jordan, A. Oberg, and M. Pollicott, \textit{Multifractal analysis of non-uniformly
hyperbolic systems}. Isreal J. Math., Vol 177, 125-144, (2008).
\bibitem{Jordan Simon}
T. Jordan and K. Simon, \textit{Multifractal Analysis of Birkhoff Averages for some Self-Affine IFS}. Dynamical Systems, Vol 22, Issue 4, (2007), 469-483.
\bibitem{LYMED2}
F. Ledrappier and L. S. Young, \textit{The Metric Entropy of Diffeomorphisms: Part II: Relations between Entropy, Exponents and
Dimension}, Ann. of Math. 122 (1985)
509-574.
\bibitem{McMullen}
C. McMullen, \textit{The Hausdorff Dimension of General Sierpinski Carpets}. Nagoya Maths Journal, vol. 96 (1984).
\bibitem{Nielsen}
O. Nielsen, \textit{The Hausdorff and packing dimensions of some sets related to Sierpinski carpets}. Canad. J. Math. 51 (1999), 1073-1088.
\bibitem{Olsen Multifractal 1}
L. Olsen, \textit{Multifractal analysis of divergence points of deformed measure theoretical Birkhoff averages}. J. Math. Pures Appl. (9) 82 (2003), no. 12, 1591--1649.
\bibitem{Olsen Winter}
L. Olsen and S. Winter, \textit{Multifractal analysis of divergence points of deformed measure theoretical Birkhoff averages. II. Non-linearity, divergence points and Banach space valued spectra}. Bull. Sci. Math. 131 (2007), no. 6, 518--558.
\bibitem{Pesin}
Y. Pesin, \textit{Dimension Theory in Dynamical Systems. Contemporary views and applications}. Chicago Lectures in Mathematics. University of Chicago Press, Chicago, IL, (1997).
\bibitem{Pesin Weiss Birkhoff}
Y. Pesin and H. Weiss, \textit{The Multifractal Analysis of Birkhoff Averages and Large Deviations}. Global Analysis of Dynamical Systems, 419-431. Inst. Phys., Bristol, (2001).
\bibitem{Walters}
P. Walters, \textit{An Introduction to Ergodic Theory}. Graduate Texts in Mathematics, 79. Springer-Verlag, New York-Berlin, (1982).
\end{thebibliography}
\end{document}